\documentclass[12pt,twoside,reqno]{amsart}

\usepackage[colorlinks=true,citecolor=blue]{hyperref}
\usepackage{mathptmx, amsmath, amssymb, amsfonts, amsthm, mathptmx, enumerate, color,mathrsfs}
\usepackage{graphicx}

\usepackage{multirow}
\usepackage{multicol}
\usepackage{algorithm}
\usepackage{algorithmic}
\usepackage{epstopdf}
\usepackage{tikz}

\newtheorem{theorem}{Theorem}[section]

\newtheorem{proposition}[theorem]{Proposition}

\theoremstyle{definition}
\newtheorem{definition}[theorem]{Definition}

\newtheorem{example}[theorem]{Example}

\newtheorem{remark}[theorem]{Remark}
\numberwithin{equation}{section}

\newcommand{\xad}{x_\alpha^\delta}

\newcommand{\bettertextcircled}[1]{\raisebox{0.9pt}{\textcircled{\raisebox{-0.9pt}{#1}}}}

\begin{document}
\setcounter{page}{1}

\vspace*{2.0cm}
\title[New aspects of ill-posedness classification in Banach spaces]
{New aspects of ill-posedness classification in Banach spaces}
\author[J. Flemming, B. Hofmann]{Jens Flemming$^1$, Bernd Hofmann$^2$}
\date{November 6, 2025}
\maketitle
\vspace*{-0.6cm}

\begin{center}
{\footnotesize

$^1$Faculty of Physical Engineering/Computer Sciences,\\
Zwickau University of Applied Sciences, 08056 Zwickau, Germany\\
$^2$Faculty of Mathematics, Chemnitz University of Technology, 09107 Chemnitz, Germany
}\end{center}

\vskip 4mm {\footnotesize \noindent {\bf Abstract.}
Motivated by a seminal paper of professor M.~Z. Nashed published in 1987 on classification of ill-posed linear operator equations and distinguishing two types of ill-posedness in Banach and Hilbert spaces, we present, illustrate and justify a new classification scheme in this context. This scheme classifies bounded linear operators mapping between infinite-dimensional Banach spaces
with respect to ill-posedness types, including non-injective operators that may have uncomplemented null-spaces. The hybrid case of strictly singular operators the range of which contains a closed infinite-dimensional subspace plays a prominent role there. By a series of new theorems we complement moreover the theory of $\ell^1$-regularization with respect to ill-posedness phenomena and shed some light on the role of weak*-to-weak continuity in the context of $\ell^1$-regularization for operators with uncomplemented null-space.


  }

\renewcommand{\thefootnote}{}
\footnotetext{
E-mail addresses: jens.flemming@htw-dresden.de (J. Flemming) and hofmannb@mathematik.tu-chemnitz.de (B. Hofmann).
}

\section{Introduction} \label{sec:introduction}

Let $X$ and $Y$ be \emph{infinite-dimensional Banach spaces} and let $A: X \to Y$ denote a \emph{bounded linear operator} mapping between $X$ and $Y$.
Our focus is on the associated linear operator equation
\begin{equation}\label{eq:opeq}
A\, x =y  \quad (x \in X,\; y \in Y)\,,
\end{equation}
assuming that this equation has to be considered to be \emph{ill-posed} is some well-defined sense.

In the seminal paper \cite{Nashed86}, M.~Z.~Nashed introduced a new approach to \emph{classification} and \emph{distinguishing type~I and type~II of ill-posedness} for such equations \eqref{eq:opeq} and thus of the corresonding operators $A$. For the article \cite{Nashed86} and a wide range of publications on linear problems in abstract function spaces since then, ill-posedness has been characterized by a \emph{non-closed range} $\mathcal{R}(A)$ of the operator $A$, whereas a closed range of $A$ indicates well-posedness. The well-posedness in case of a closed range is motivated by the fact that a closed range implies \emph{stability} in the sense of Ivanov \cite{Ivanov63} also for non-injective $A$. This means, for the exact right-hand side $y \in \mathcal{R}(A)=\overline{\mathcal{R}(A)}^Y$, that approximations $y_n \in \mathcal{R}(A)$ with
$\lim \limits_{n \to \infty} \|y_n-y\|_Y=0$ imply the convergence of the quasi-distance ${\rm qdist}(A^{-1}(y_n),A^{-1}(y))$ (cf.~\cite[Def.~1]{HofPla18}) to zero as $n \to \infty$
(see, e.g.,~\cite[Prop.~1.12]{Flemmingbuch18}).

Even though the concept of complementedness of null-spaces $\mathcal{N}(A)$ of the operator $A$ already appeared around ill-posedness types by Nashed in \cite[Thm.~4.5]{Nashed86}, the detailed influence of \emph{complemented} or \emph{uncomplemented} null-space $\mathcal{N}(A)$ on the characterization of ill-posedness was seriously addressed for the first time in the monograph \cite{Flemmingbuch18} of the first author. The simplification that \eqref{eq:opeq} is well-posed if the range is closed in $Y$ and ill-posed if  $\mathcal{R}(A)\not=\overline{\mathcal{R}(A)}^Y$ is fully consistent if $X$ and $Y$ are both \emph{Hilbert spaces} and also in the Banach space case if $A$ is \emph{injective}. In this context, we refer for details to our earlier article \cite{FHV15}, and in particular to Figure~1 therein.

If, however, $A$ is a non-injective linear operator in general Banach spaces, uncomplemented null-spaces $\mathcal{N}(A)$ in $X$ may occur. For uncomplemented null-spaces it is clear that
a closed range implies stability in the sense of Ivanov, but it is not clear whether that stability fails if the range is not closed. But much more important for uncomplemented null-spaces $\mathcal{N}(A)$ is that occurring pseudoinverses of $A$ are always unbounded (see for consequences Proposition~\ref{pro:pseudoinverse} below), and we remember in this context that we know unboundedness as an ill-posedness criterion with respect to the  Moore-Penrose pseudoinverse in the Hilbert space setting.
Taking into account the interplay of uncomplemented and complemented null-spaces, the characterization and classification of ill-posedness for both injective and non-injective operators $A$ in Banach spaces was intensively discussed in a recent paper of the second author with S.~Kindermann \cite{HofKin25}, and we refer for a comparision with our following studies preferably to Section~4 ibid with the Definition~4.1 and its illustration by Figure~3.

In \cite{HofKin25}, the occurrence of the so-called \emph{hybrid case} was discovered, i.e.~the existence of hybrid-type operators $A$ in the sense of Definition~\ref{def:hybrid} below, where the range $\mathcal{R}(A)$ of a strictly singular operator contains a closed infinite-dimensional subspace. This is only possible if the null-space $\mathcal{N}(A)$ is uncomplemented and $A$ is not compact (see Proposition~\ref{pro:hybridprop} below). Typical representatives for hybrid-type operators are bounded linear operators $A=A_{\mathrm{Maz}}$ mapping the Banach space $X$ \emph{onto} a separable Banach space $Y$ that we will call \emph{Mazur-type operators}, see Example~\ref{ex:1} for details. The hybrid case and Mazur-type operators are of particular interest to our paper.

On the one hand, we will suggest a partially new Definition~\ref{def:new} for well- and ill-posedness characterization and classification as an alternative to \cite[Def.~4.1]{HofKin25}, where as main point in contrast to the former definition the ill-posedness type of hybrid cases switches  from type~II to type~I such that now all ill-posed operators $A$ containing a closed infinite-dimensional subspace are uniformly of type~I. This new definition collects as ill-posed of type~II all equations \eqref{eq:opeq} (area to the right of the vertical line in Figure~\ref{fig:update} below) that cannot be saved in the sense that partial overcoming of ill-posedness through domain restriction in the inversion process or finding continuous nonlinear pseudoinverses (possible in one way or another for type~I ill-posendess) completely fails for type~II.

On the other hand, we will give examples that equations with hybrid-type operators $A$ should not be considered and handled as well-posed problems, although  $\mathcal{R}(A)$ may even be closed.
These reasons are closely releated to deficits of the method of $\ell^1$-regularization for operators of the form $A: \ell^1 \to Y$ in the case of ill-posedness of type~I.
The new Theorem~\ref{thm:newl1} below will show the failure of the weak*-weak continuity of such operators. This kind of continuity of $A$, however, is an important sufficient condition for existence, stability and convergence of
$\ell^1$-regularized solutions (see Proposition~\ref{pro:exist} below). For the specific example of a Mazur-type operator that maps $\ell^1$ onto $\ell^2$, the failure of the $\ell^1$-regularization can even be outlined in an explicit manner.

The paper is organized as follows: In Section~\ref{sec:classification} we present, illustrated by a figure as well as by examples, and justify our new classification scheme with shifted hybrid case. In this context, we introduce Mazur-type operators. These operators are of hybrid-type and play a prominent role in the examples and in Section~\ref{sec:ell1}, in which we discuss phenomena of $\ell^1$-regularization and ill-posedness. Section~\ref{sec:ell1} presents and proves a proposition and three new theorems around the weak$^*$-to-weak continuity of operators in $\ell^1$, its failure for operators ill-posed of type~I and the associated consequences.

\section{A new facet for ill-posedness classification in Banach spaces} \label{sec:classification}

\subsection{Preliminaries} \label{sec:prel}

Since the null-space $\mathcal{N}(A)$ of the bounded linear operator $A$ is a closed linear subspace of $X$, we always find another subspace $U$ in $X$ such that
$$X=\mathcal{N}(A)\oplus U \quad \mbox{is a direct sum of} \;\; \mathcal{N}(A)\;\mbox{and}\;\,U,  $$
but such algebraic complements $U$ of $\mathcal{N}(A)$ need not be unique. If, in particular, there exists a \emph{closed} complement $U$ of $\mathcal{N}(A)$, then the null-space is called topologically complemented, or simply \emph{complemented} in $X$, otherwise \emph{uncomplemented} in $X$.
Banach spaces $X$, which are not isomorphic  to a Hilbert space, always contain uncomplemented subspaces (see, e.g.,~\cite{LiTz77}). Consequently, for such Banach spaces $X$,
uncomplemented null-spaces $\mathcal{N}(A)$ in $X$ may occur for bounded linear operators $A: X \to Y$.
In separable Banach spaces each closed subspace, including the uncomplemented ones, is the null-space of some bounded linear operator with $X=Y$, cf.\ \cite[Prop.~2.1]{LauWhi20}. The quotient map $A:X\rightarrow X/U$ corresponding to an uncomplemented subspace $U$ is another example of a bounded linear operators with uncomplemented null-space.
Along the lines of discussions in \cite[Section~1]{Flemmingbuch18} we mention that, for $X$ with a direct sum as $X=\mathcal{N}(A)\,\oplus \, U$, the restriction
$A|_U: U \to \mathcal{R}(A) \subset Y$ of $A$ is a bijective mapping, and we denote its well-defined inverse
$A^\dagger_U: \mathcal{R}(A)  \to U$ as \emph{pseudoinverse} of $A$ with respect to $U$.

\begin{proposition} \label{pro:pseudoinverse}
If for the bounded linear operator $A:X \to Y$ from equation \eqref{eq:opeq} the null-space $\mathcal{N}(A)$ is \emph{complemented} in $X$ with the closed infinite-dimensional subspace $U$ of the Banach space $X$ such that the direct sum $X=\mathcal{N}(A)\oplus U$ takes place, then the pseudoinverse $A^\dagger_U: \mathcal{R}(A)  \to U$ is a \emph{bounded} linear operator if and only if the range $\mathcal{R}(A)$ is closed. If, however,
$\mathcal{N}(A)$ is \emph{uncomplemented} in $X$ and hence every infinite-dimensional subspace $U$ with $X=\mathcal{N}(A)\oplus U$ is not closed, then
the pseudoinverse $A^\dagger_U: \mathcal{R}(A)  \to U$ is always \emph{unbounded}.
\end{proposition}
\begin{proof}
For complemented null-spaces, the assertion of the proposition is a consequence of the open mapping theorem, whereas the assertion for uncomplemented null-spaces can be found with proof in \cite[Prop.~1.11]{Flemmingbuch18}.
\end{proof}

For the classification of well-posed and ill-posed equations \eqref{eq:opeq}, it plays a prominent role on the one hand whether the associated bounded linear operators $A$
have a finite-dimensional or infinite-dimensional range $\mathcal{R}(A)$ and on the other hand if the operators are strictly singular or not and if the operators are compact or not.
We recall that $A:X \to Y$ is called \emph{strictly singular} if, given any closed infinite-dimensional subspace $Z$ of $X$, $A$ restricted to $Z$ is not an isomorphism (i.e., linear homeomorphism), which means that closed subspaces $Z$ of $X$, for which the restriction $A|_{Z}$ has a bounded inverse, are necessarily finite dimensional. We also note that
all compact operators are strictly singular and that the following proposition taken from \cite[Lemma~1.4]{HofKin25} is valid.

\begin{proposition} \label{pro:lemmaold}
If the range $\mathcal{R}(A)$ of the bounded linear operator $A: X \to Y$ mapping between infinite-dimensional Banach spaces $X$ and $Y$ is finite-dimensional, then $A$ is strictly singular and compact with closed range $\mathcal{R}(A)$. Conversely, let $A$  be strictly singular and possess a closed range $\mathcal{R}(A)$. Then $\mathcal{R}(A)$ is finite-dimensional and $A$ is also compact whenever either
\begin{itemize}
\item $X$ and $Y$ are Hilbert spaces
\item or $A$ is an injective mapping between Banach spaces $X$ and $Y$.
\end{itemize}
\end{proposition}

The situation becomes more complicated if non-injective operators $A$ and uncomplemented nullspaces $\mathcal{N}(A)$ are taken into account. Then hybrid-type operators many occur, and
along the lines of \cite[Def.~4.6 and Prop.~4.7]{HofKin25} we give the following Definition~\ref{def:hybrid} and Proposition~\ref{pro:hybridprop}.

\begin{definition}[Hybrid-type] \label{def:hybrid}
We characterize the operator equation \eqref{eq:opeq} and its corresponding bounded linear operator $A: X \to Y$  as of \emph{hybrid-type} if $A$ is strictly singular and its range $\mathcal{R}(A)$ contains an infinite-dimensional closed subspace of $Y$.
\end{definition}

\begin{proposition} \label{pro:hybridprop}
For an operator equation \eqref{eq:opeq} of hybrid-type, the operator $A: X \to Y$ is not compact and its null-space $\mathcal{N}(A)$ is always uncomplemented.
\end{proposition}
\begin{proof}
By Definition~\ref{def:hybrid}, $\mathcal{R}(A)$ contains a closed infinite-dimensional subspace, say $M$. Since $M$ is infinite-dimensional, then by the Riesz lemma there is a bounded sequence $(y_n) \subset M$ which has no convergent subsequence. Even though $A$ is non-injective, the
open mapping theorem in the formulation of \cite[Theorem~3]{Tao09b} still gives a sequence $(x_n)_{n\in\mathbb{N}}$ with $x_n \in A^{-1}(y_n),\;$ $\|x_n\|_X \leq C \|y_n\|_Y$ and
$A x_n = y_n$ for all $n\in\mathbb{N}$. If $A$ was compact, then $(y_n)_{n\in\mathbb{N}}$ would have a convergent subsequence, which is a contradiction.

If the null-space would be complemented and $A$ is not compact,
then $X=\mathcal{N}(A) \oplus U$ is connected with an infinite-dimensional closed subspace $U$ of $X$. This, however, leads to a contradiction. Namely, let $M \subseteq \mathcal{R}(A)$ be an
infinite-dimensional closed subspace of $Y$ and, due to the continuity and injectivity of $A$ on $U$, the set $A_U^{\dagger}[M]$
is an infinite-dimensional closed subspace of $X$ on which $A$ is continuously invertible as a consequence of
Proposition~\ref{pro:pseudoinverse}.
This, however, contradicts the strict singularity of $A$.
\end{proof}

\begin{example}[Mazur-type operators in sequence spaces] \label{ex:1}
Let $X:=\ell^1$ and let $Y$ be a separable Banach space. Further, let $(\zeta^{(k)})_{k\in\mathbb{N}}$ be a countable dense subset of the unit sphere in $Y$ with $\zeta^{(k)}\neq\zeta^{(l)}$ for $k\neq l$. Then, for $x=(x_1,x_2,...) \in \ell^1$,
\begin{equation}\label{eq:mazur}
A_{\mathrm{Maz}}\,x:=\sum_{k=1}^\infty x_k\,\zeta^{(k)}
\end{equation}
defines a bounded linear operator mapping $\ell^1$ \emph{onto} $Y$ (\cite[Theorem~2.3.1 and its proof]{AlbKal06}). If $Y$ is not isomorphic to a subspace of $\ell^1$, then the null-space $\mathcal{N}(A_{\mathrm{Maz}})$ is uncomplemented (\cite[Corollary 2.3.3 and its proof]{AlbKal06}).
Note that choosing different separable Banach spaces $Y$ and different sequences $(\zeta^{(k)})_{k\in\mathbb{N}}$ we obtain many different operators with uncomplemented null-space. We refer to such operators as \emph{Mazur-type} operators, because their existence was already shown in \cite[page~111, item e)]{Mazur33}.
\par
Of particular interest to us are Mazur-type operators $A_{\mathrm{Maz}}:\ell^1\to\ell^q$, $1<q<\infty$, mapping the non-reflexive Banach space $\ell^1$ continuously \emph{onto} the reflexive and separable Banach space $\ell^q$.
Such operators are strictly singular (see \cite[Theorem]{GoldThorp63}), because $\ell^1$ does not contain an infinite-dimensional
reflexive subspace and $\ell^q$ is a reflexive Banach space. Moreover, as a consequence of Proposition~\ref{pro:hybridprop}, such an operator $A_{\mathrm{Maz}}$ is \emph{not compact} and possesses an \emph{uncomplemented null-space}. Our focus is on the special case $q=2$ introduced as example in \cite{GoldThorp63}. In the sequel, we will abbreviate this operator mapping from $\ell^1$ on $\ell^2$ as $B$.
\end{example}

\subsection{Updated classification with switched hybrid case and its illustration} \label{sec:def}

\begin{definition}[Well- and ill-posedness characterization and classification] \label{def:new}
Let $A: X \to Y$ be a bounded linear operator mapping between the
infinite-dimensional Banach spaces $X$ and $Y$.

Then the operator equation \eqref{eq:opeq} is called \emph{well-posed} if
\begin{align*} &\text{the range $\mathcal{R}(A)$ of $A$ is a \emph{closed} subset of $Y$ and, moreover,} \\
&\text{the null-space $\mathcal{N}(A)$ is \emph{complemented} in $X$;}
\end{align*}
otherwise the equation (\ref{eq:opeq}) is called \emph{ill-posed}.

\medskip

In the ill-posed case, (\ref{eq:opeq}) is called \emph{ill-posed of type I}   if
\begin{align*}
&\text{the range $\mathcal{R}(A)$ contains an \emph{infinite-dimensional closed subspace} of $Y$;}
\end{align*}
otherwise the ill-posed
equation (\ref{eq:opeq}) is called \emph{ill-posed of type~II}.
\end{definition}

\begin{figure}[ht]
\begin{center}
\includegraphics{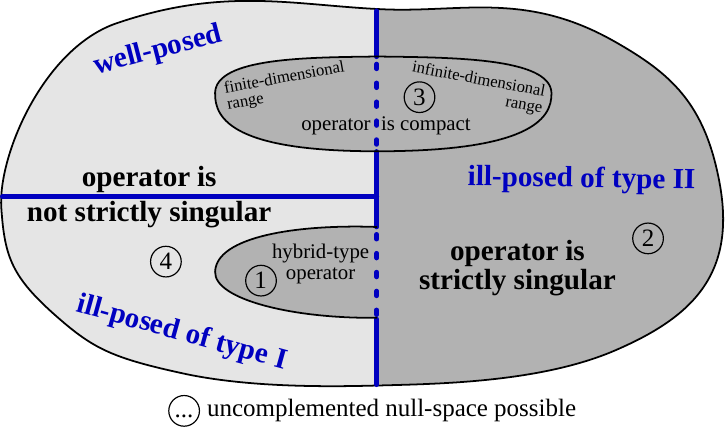}
\caption{Case distinction for bounded linear operators between
infinite-dimensional Banach spaces with complemented and uncomplemented null-spaces
and shifted hybrid case} \label{fig:update}
\end{center}
\end{figure}


\subsection{Discussion of Figure~\ref{fig:update} } \label{sec:discussion}

The figure should help us to illustrate the impact of Definition~\ref{def:new} on the classification of well-posedness and ill-posedness of the linear operator equation \eqref{eq:opeq}
formulated in Banach spaces $X$ and $Y$, where the operator $A: X \to Y$ may be non-injective and may have an uncomplemented null-space $\mathcal{N}(A)$. Bearing in mind the operator
properties described in the preliminaries, we are going to discuss in this subsection the different possible cases with reflection to the areas appearing in the figure.

The area to the right of the vertical line in the figure is devoted to the equations that are ill-posed of type~II in the sense of Definition~\ref{def:new} with strictly singular operators $A$ possessing an infinite-dimensional range $\mathcal{R}(A)$,
which has no closed infinite-dimensional subspace. As we already had mentioned in the introduction, those problems cannot be saved in the sense that partial
overcoming of ill-posedness through domain restriction in the inversion process or finding continuous nonlinear pseudoinverses via the  Bartle--Graves theorem (cf.\ \cite[Theorem~5.5]{HofKin25}) is possible. This is just because their ranges $\mathcal{R}(A)$ do not contain an infinite-dimensional
closed subspace. For a Hilbert space setting, this right area if full of compact operators $A$. In the Banach space setting, however, the compact operators form only a proper subset.
For example, in sequence spaces the injective embedding operators $A:=\mathcal{E}_p^q: \ell^p \to \ell^q\;(1 \le p<q<\infty)$ are strictly singular but not compact, and we refer for related problems also
to the recent paper \cite{PlaHof25}.

The area to the left of the vertical line in the figure is separated by a horizontal line. Above the line one can find the well-posed problems in the sense of Definition~\ref{def:new} and below the line the ill-posed problems of type~I in the sense of that definition.
From the right side, two protrusions (bulgs) grow into this left area. Both express strictly singular operators, whereas otherwise all operators $A$ belonging to the area left of the vertical line are not strictly singular. Strictly singular operators correspond to the areas with dark background in the figure. The upper protrusion collects the well-posed problems with operators $A$ possessing finite-dimensional ranges, whereas the lower protrusion expresses the interesting hybrid case in the sense of Definition~\ref{def:hybrid}. The figure also shows that hybrid-type operators are
well-separated from the compact operators.

It seems to be of interest, in which areas of Figure~\ref{fig:update} operators $A$ with \emph{uncomplemented} null-spaces $\mathcal{N}(A)$ can be found. We have indicated in the figure such areas by the small circled numbers \bettertextcircled{1}, \bettertextcircled{2}, \bettertextcircled{3} and \bettertextcircled{4}. The number \bettertextcircled{1} meets the hybrid case (see also Example~\ref{ex:1} above), and it should be emphasized that \emph{all} operators $A$ of the hybrid case have uncomplemented null-spaces as Proposition~\ref{pro:hybridprop} proves. In the three areas, where \bettertextcircled{2}, \bettertextcircled{3} and \bettertextcircled{4} appear, uncomplemented null-spaces are not typical but they can appear as outlined in Examples~\ref{ex:2}, \ref{ex:3} and \ref{ex:4} below, where the Mazur-type operator $B$ (see Example~\ref{ex:1}) mapping from $\ell^1$ onto $\ell^2$ with null-space uncomplemented in $\ell^1$ is an auxiliary tool for the construction of concrete versions to all three examples. In the area of well-posed problems, uncomplemented null-spaces are excluded by definition. Neither operators $A$ with closed infinite-dimensional range and continuous inverse
$A^{-1}: \mathcal{R}(A)=\overline{\mathcal{R}(A)}^Y \to X$ nor compact operators $A$ with finite-dimensional range may have uncomplemented null-spaces.

\begin{example}[Non-compact compositions to Mazur-type operator $B$] \label{ex:2}
 \bettertextcircled{2} As an example for strictly singular but non-compact operators $A$ with infinite-dimensional range and \emph{uncomplemented null-space} indicated by number (2) we introduce a composition
$A=:C \circ B: \ell^1 \to Z$ of a version of the Mazur-type operator $B: \ell^1 \to \ell^2$, possessing an uncomplemented null-space $\mathcal{N}(B)$, and a bounded non-compact injective operator $C: \ell^2 \to Z$ with a non-closed range $\mathcal{R}(C)$ that does not contain a closed infinite-dimensional subspace. Due to the surjectivity of $B$ and the injectivity of $C$, the range $\mathcal{R}(A)$ of $A$ is then also non-closed and  does not contain a closed infinite-dimensional subspace, moreover the null-space $\mathcal{N}(A)$ is also \emph{uncomplemented}.
A concrete version of such operator $C$ is the strictly singular and non-compact embedding operator $C=\mathcal{E}_2^p: \ell^2 \to Z=\ell^p \;(2<p<\infty)$.
Non-compactness of $C$ yields a bounded sequence $(y_n)_{n\in\mathbb{N}} \subset Y$ such that $(C\,y_n)_{n\in\mathbb{N}} \subset Z$ does not contain a convergent subsequence.
In the definition \eqref{eq:mazur} of $B$ we choose a sequence $(\zeta^{(k)})_{k\in\mathbb{N}}$ that contains all $\tilde{y}_n:=\frac{1}{\|y_n\|_Y}y_n$.
If $(k_n)_{n\in\mathbb{N}}$ is such that $\zeta^{k_n}=\tilde{y}_n$, then we have $x^{(n)}:=\|y_n\|_Y\,e^{(k_n)}$, with $e^{(k)}:=(0,\ldots,0,1,0,\ldots)$ being the usual unit sequences ($1$ at position $k$). This implies $B\,x^{(n)}=y_n$, and thus the image of the bounded sequence $(x^{(n)})_{n\in\mathbb{N}}$ under the composition $A=C\circ B$ does not contain any convergent subsequence. Consequently, $A$ is not compact, but evidently also not of hybrid-type.
\end{example}

\begin{example}[Compact compositions to Mazur-type operator $B$] \label{ex:3}
 \bettertextcircled{3} As an example for compact operators $A$ with infinite-dimensional range and \emph{uncomplemented null-space} indicated by number (3) we can present a composition operator $A:=C \circ B$ for a  compact injective operator $C: \ell^2 \to Y$ and the Mazur-type operator $B: \ell^1 \to \ell^2$ from Example~\ref{ex:1}. Evidently, the composition is \emph{compact}, because $B$ is bounded. Since the null-space of $B$ is uncomplemented in $\ell^1$ and $C$ is injective, the null-space of $A$ is also \emph{uncomplemented}.
\end{example}

\begin{example}[Type~I ill-posedness in product spaces] \label{ex:4}
 \bettertextcircled{4} Here, we consider the product spaces $X:=\ell^1 \times \ell^2, \;Y:=\ell^2 \times \ell^2$ and the bounded linear operator $A:=(B,I)$ with the Mazur-type operator $B: \ell^1 \to \ell^2$ from Example~\ref{ex:1} and the identity operator $I: \ell^2 \to \ell^2$. Then the null-space $\mathcal{N}(A)=\mathcal{N}(B) \times \{0\}$ is \emph{uncomplemented} since $\mathcal{N}(B)$ is,
which means that \eqref{eq:opeq} with such operator $A$ is ill-posed.
Obviously, $A$ is an isomorphism on $\{0\}\times \ell^2$ and therefore not strictly singular, consequently not of hybrid-type. Moreover, $A$ is \emph{ill-posed of type~I}, because the range $\mathcal{R}(A)$  contains the infinite-dimensional closed subspace $M=\{0\}\times \ell^2$.
\end{example}

\begin{remark} \label{rem:1}
As the repeated use of the operator $B$ in the three examples above subtly suggests, the occurrence of bounded linear operators with uncomplemented null-spaces seems to be rare, and their targeted construction appears to be challenging. Now let $X$ be a Banach space with a closed subspace $S$ uncomplemented in $X$ and $U$ a subspace such that the direct sum $X=S \oplus U$ takes place.
One might then consider it a simple construction to define an operator $A: X \to X$ with uncomplemented null-space by the formula
$$
Ax:=\begin{cases}
0,&\text{if }x\in S,\\
I,&\text{if }x\in U,
\end{cases}
$$
with $I:U\to U$ being the identity map on $U$.
This is obviously a linear operator with null-space $\mathcal{N}(A)=S$, but unfortunately such operator is always unbounded. This is because $P:=I-A$ is a bounded projection from $X$ onto $S$ whenever $A$ is a bounded operator, and this is only possible if $S$ is complemented in $X$, see \cite[Cor.~3.2.15]{Meg98}.
\end{remark}

\begin{remark} \label{rem:2}
As also Figure~\ref{fig:update} indicates, the class of linear bounded operators $A$ mapping between Banach spaces that are ill-posed in the sense of Definition~\ref{def:new} is divided into two clearly separated subclasses of type~I and type~II ill-posedness. Surprisingly, the dividing line between these two types can be crossed when observing compositions of two operators of the same type. As an example we consider, for some Banach space $X$ and some strictly singular operator $C: X \to X$ with non-closed range, the two operators $D_1:=(C,I)$ and $D_2:=(I,C)$ both mapping from $X \times X$ into itself, where $I$ is the identity operator in $X$. These operators are ill-posed of type~I, because their ranges $\mathcal{R}(D_1)$ and  $\mathcal{R}(D_2)$ are non-closed and contain closed infinite dimensional subspaces in the Banach space $X \times X$. However, their composition $A:=D_2 \circ D_1=(C,C)$ is strictly singular with non-closed range and therefore ill-posed of type~II.

A similar phenomenon had been observed in \cite{KinHof24} for Hilbert spaces, where also the dividing line between the well-separated compact (type~II) and non-compact (type~I) ill-posed operators can be crossed by composition. Ibid there was analyzed the case that the composition of two non-compact operators can lead to a compact one.
\end{remark}

\section{Phenomena of $\ell^1$-regularization and ill-posedness} \label{sec:ell1}

In this section, we consider the operator equation \eqref{eq:opeq} in the special case of bounded linear operators
\begin{equation} \label{eq:Al1}
A: \ell^1 \to Y\,,
\end{equation}
mapping from $\ell^1$ to the infinite-dimensional Banach space $Y$. In most cases such equations are ill-posed, and
one way to obtain stable approximate solutions to operator equations in the $\ell^1$-case is to minimize the Tikhonov-type functional
\begin{equation} \label{eq:Tikl1}
T_\alpha^\delta(x):= \|Ax-y^\delta\|_Y^p+ \alpha\,\|x\|_{\ell^1}
\end{equation}
over $x \in \ell^1$ with regularization parameter $\alpha>1$ and some appropriate exponent $p>1$. An important sufficient condition for the existence and stability of minimizers $\xad \in \ell^1$
to the functional $T_\alpha^\delta(x)$ is the weak$^*$-to-weak continuity of the operator $A$ from \eqref{eq:Al1} in the sense of Proposition~\ref{pro:exist}, for which we refer to \cite[Theorem~9.4]{Flemmingbuch18}. Here, weak$^*$-to-weak continuity means that $A$ transforms weak$^*$-convergent sequences in $\ell^1$ into weak-convergent sequences in $Y$, where the weak$^*$-convergence in $\ell^1$ is related to the predual space $c_0$ of $\ell^1$.

\begin{proposition} \label{pro:exist}
Let $A:\ell^1 \to Y$ be weak$^*$-to-weak continuous. Then the following assertions are true.
\begin{itemize}
\item[(i)]
Existence: There exist solutions to \eqref{eq:opeq} with minimal norm (referred to as norm minimizing solutions) and there exist minimizers of the Tikhonov-type functional \eqref{eq:Tikl1}.
Further, all minimizers of $T_\alpha^\delta$ are sparse.
\item[(ii)]
Stability: If $(y_k)_{k\in\mathbb{N}}$ converges to $y^\delta$ and if $(x^{(k)})_{k\in\mathbb{N}}$ is a corresponding sequence of minimizers of \eqref{eq:Tikl1} with $y^\delta$ replaced by $y_k$, then this second sequence has a weakly* convergent subsequence and each weakly* convergent subsequence converges weakly* to a minimizer of $T_\alpha^\delta$.
\item[(iii)]
Convergence: If $(\delta_k)_{k\in\mathbb{N}}$ converges to zero and if $(y_k)_{k\in\mathbb{N}}$ satisfies $\|y_k-y^\dagger\|_Y\leq\delta_k$, then there is a sequence $(\alpha_k)_{k\in\mathbb{N}}$ such that each corresponding sequence of minimizers of $T_{\alpha_k}^{\delta_k}$ contains a weakly* convergent subsequence. Each such subsequence converges in norm to some norm minimizing solution of \eqref{eq:opeq}.
\end{itemize}

\end{proposition}

The following theorem reflects the weak$^*$-to-weak continuity assumption of Proposition~\ref{pro:exist} in light of the type-classification from Definition~\ref{def:new}.

\begin{theorem} \label{thm:newl1}
If the bounded linear operator $A: \ell^1 \to Y$ has infinite-dimensional range and is weak$^*$-to-weak continuous, then the associated operator equation \eqref{eq:opeq} is ill-posed of type~II
in the sense of Definition~\ref{def:new}.
\end{theorem}
\begin{proof}
See \cite[Theorem~10.5]{Flemmingbuch18} for a proof. There, the unmodified definition for type-I ill-posedness is used, that is, type I requires that the range of $A$ contains a closed infinite-dimensional subspace $M$ and that $\mathcal{N}(A)$ is complemented in the full pre-image $A^{-1}[M]$. But the proof solely relies on the first condition and does not use any arugment related to (un)complementedness of $\mathcal{N}(A)$, thus, fits our modified Definition~\ref{def:new}, too.
\end{proof}

Weak*-to-weak continuity is a sufficient but not a necessary condition for existence, stability, and convergence of Tikhonov minimizers. Especially for Mazur-type operators, which always have a closed range, there is some chance that Tikhonov regularized solutions exist, are stable, and converge although Mazur-type operators lack weak*-to-weak continuity. But the next theorem shows that Tikhonov regularization fails for Mazur-type operators.

\begin{theorem} \label{thm:l1failure}
Let $A:=B:\ell^1\rightarrow\ell^2$ be the Mazur-type operator defined in Example~\ref{ex:1} based on a sequence $(\zeta^{(k)})_{k\in\mathbb{N}}$.
For $y^\delta\in\ell^2$ and $\alpha\geq 0$ the $\ell^1$-Tikhonov functional \eqref{eq:Tikl1} has minimizers if and only if $y^\delta=\lambda\,\zeta^{(k)}$ for some $\lambda\in\mathbb{R}$ and some $k\in\mathbb{N}$. For such $y^\delta$ there is exactly one minimizer if $\zeta^{(l)}\neq-\zeta^{(k)}$ for all $l\in\mathbb{N}$, $l\neq k$. The minimizer attains the form
\begin{equation}
x_\alpha^\delta=\begin{cases}
0,&\text{if }\lambda\in[-\alpha,\alpha],\\
(\lambda-\alpha)\,e^{(k)},&\text{if }\lambda>\alpha,\\
(\lambda+\alpha)\,e^{(k)},&\text{if }\lambda<-\alpha.
\end{cases}
\end{equation}
If $\zeta^{(k)}=-\zeta^{(l)}$ for some $l$, then for each
\begin{equation}
\gamma\in\begin{cases}
(\alpha-\lambda,0),&\text{if }\lambda>\alpha,\\
(0,-\alpha-\lambda),&\text{if }\lambda<-\alpha
\end{cases}
\end{equation}
there is an additional minimizer of the form
\begin{equation}
x_\alpha^\delta=\begin{cases}
(\gamma+\lambda-\alpha)\,e^{(k)}+\gamma\,e^{(l)},&\text{if }\lambda>\alpha,\\
(\gamma+\lambda+\alpha)\,e^{(k)}+\gamma\,e^{(l)},&\text{if }\lambda<-\alpha.
\end{cases}
\end{equation}
Here $e^{(k)}$ denotes the sequence $(0,\ldots,0,1,0,\ldots)$ with $1$ at position $k$.
\end{theorem}

\begin{proof}
From
\begin{equation}
\langle\eta,A\,x\rangle=\sum_{k\in\mathbb{N}}x_k\,\langle\eta,\zeta^{(k)}\rangle
\end{equation}
for $\eta\in\ell^2$ and $x\in\ell^1$ we see
\begin{equation}
A^\ast\eta=(\langle\eta,\zeta^{(1)}\rangle,\langle\eta,\zeta^{(2)}\rangle,\ldots)\in\ell^\infty.
\end{equation}
Due to convexity and continuity of the Tikhonov functional, $x_\alpha^\delta\in\ell^1$ is a minimizer of $T_\alpha^\delta$ if and only if
\begin{equation}
-\tfrac{1}{\alpha}\,A^\ast\,(A\,x_\alpha^\delta-y^\delta)\in\partial\|\cdot\|_{\ell^1}(x_\alpha^\delta),
\end{equation}
where
\begin{equation}
\xi\in\partial\|\cdot\|_{\ell^1}(x)
\quad\Leftrightarrow\quad
\xi_k\begin{cases}
=-1,&\text{if }x_k<0,\\
\in[-1,1],&\text{if }x_k=0,\\
=1,&\text{if }x_k>0
\end{cases}
\;\text{for }k\in\mathbb{N}.
\end{equation}
\par
 The optimality condition implies that there is some $\eta\in\ell^2$ such that $A^\ast\eta\in\partial\|\cdot\|_{\ell^1}(x_\alpha^\delta)$. Assume that $x_\alpha^\delta$ has at least two non-zero components $[x_\alpha^\delta]_m\neq 0$ and $[x_\alpha^\delta]_n\neq 0$. Denote the signs of both components by $s_m\in\{-1,1\}$ and $s_n\in\{-1,1\}$, respectively. Then $[A^\ast\,\eta]_m=s_m$ and $[A^\ast\,\eta]_n=s_n$ or, equivalently, $\langle\eta,\zeta^{(m)}\rangle=s_m$ and $\langle\eta,\zeta^{(n)}\rangle=s_n$. Now take a subsequence $(\zeta^{(k_l)})_{l\in\mathbb{N}}$ of $(\zeta^{(k)})_{k\in\mathbb{N}}$ converging to
\begin{equation}
\tilde{\eta}:=\frac{s_m\,\zeta^{(m)}+s_n\,\zeta^{(n)}}{\|s_m\,\zeta^{(m)}+s_n\,\zeta^{(n)}\|_{\ell^2}}.
\end{equation}
Then $\langle\eta,\zeta^{(k_l)}\rangle\to\langle\eta,\tilde{\eta}\rangle$. From
\begin{equation}
\langle\eta,\tilde{\eta}\rangle=\frac{2}{\|s_m\,\zeta^{(m)}+s_n\,\zeta^{(n)}\|_{\ell^2}}.
\end{equation}
we see $\langle\eta,\tilde{\eta}\rangle\geq 1$ and that $\langle\eta,\tilde{\eta}\rangle>1$ holds if and only if $\zeta^{(m)}$ and $\zeta^{(n)}$ are linearly dependent. Thus, $\langle\eta,\tilde{\eta}\rangle=1$ is only possible for $\zeta^{(n)}=-\zeta^{(m)}$.
\par
In case $\langle\eta,\tilde{\eta}\rangle>1$ we find (large enough) $k_l$ with $\langle\eta,\zeta^{(k_l)}\rangle>1$ or, equivalently $[A^\ast\,\eta]_{k_l}>1$. Thus, $A^\ast\,\eta\notin\partial\|\cdot\|_{\ell^1}(x_\alpha^\delta)$, which shows that $x_\alpha^\delta$ can have at most one non-zero component $[x_\alpha^\delta]_k$ if $\zeta^{(l)}\neq-\zeta^{(k)}$ for all $l\neq k$.
\par
In case $\langle\eta,\tilde{\eta}\rangle=1$, we do not obtain a contradiction (at the moment). That is, $x_\alpha^\delta$ may have two non-zero components $[x_\alpha^\delta]_m$ and $[x_\alpha^\delta]_n$ as long as $\zeta^{(n)}=-\zeta^{(m)}$. But a third non-zero component $[x_\alpha^\delta]_l$ is not possible, because corresponding $\zeta^{(l)}$ would have to be equal to both $-\zeta^{(m)}$ and $-\zeta^{(n)}=\zeta^{(m)}$.
\par
Let $x_\alpha^\delta=\beta_m\,e^{(m)}$ with $\beta_m\in\mathbb{R}\setminus\{0\}$ be a minimizer with only one non-zero component. Then the optimality condition is equivalent to
\begin{equation}\label{eq:opt_beta}
\langle\beta_m\,\zeta^{(m)}-y^\delta,\zeta^{(m)}\rangle=-(\mathrm{sgn}\,\beta_m)\,\alpha
\quad\text{and}\quad
|\langle\beta_m\,\zeta^{(m)}-y^\delta,\zeta^{(l)}\rangle|\leq\alpha\;\text{for }l\neq m.
\end{equation}
The left-hand condition is equivalent to $\beta_m=\langle y^\delta,\zeta^{(m)}\rangle-(\mathrm{sgn}\,\beta_m)\,\alpha$ and, thus, to
\begin{equation}
\beta_m=\begin{cases}
\langle y^\delta,\zeta^{(m)}\rangle-\alpha,&\text{if }\langle y^\delta,\zeta^{(m)}\rangle>\alpha,\\
\langle y^\delta,\zeta^{(m)}\rangle+\alpha,&\text{if }\langle y^\delta,\zeta^{(m)}\rangle<-\alpha.
\end{cases}
\end{equation}
This shows that only
\begin{equation}
x_\alpha^\delta=\begin{cases}
0,&\text{if }\langle y^\delta,\zeta^{(m)}\rangle\in[-\alpha,\alpha],\\
(\langle y^\delta,\zeta^{(m)}\rangle-\alpha)\,e^{(m)},&\text{if }\langle y^\delta,\zeta^{(m)}\rangle>\alpha,\\
(\langle y^\delta,\zeta^{(m)}\rangle+\alpha)\,e^{(m)},&\text{if }\langle y^\delta,\zeta^{(m)}\rangle<-\alpha,
\end{cases}
\end{equation}
for each $m\in\mathbb{N}$ are candidates for minimizers. The Tikhonov functional for these candidates is
\begin{align}
T_\alpha^\delta(x_\alpha^\delta)
&=\tfrac{1}{2}\,\|\beta_m\,\zeta^{(m)}-y^\delta\|_{\ell^2}^2+\alpha\,|\beta_m|\\
&=\tfrac{1}{2}\,\beta_m^2+\tfrac{1}{2}\,\|y^\delta\|_{\ell^2}^2-\beta_m\,\langle y^\delta,\zeta^{(m)}\rangle+\alpha\,|\beta_m|\\
&=\tfrac{1}{2}\,\beta_m^2+\tfrac{1}{2}\,\|y^\delta\|_{\ell^2}^2-\beta_m\,\bigl(\langle y^\delta,\zeta^{(m)}\rangle-\alpha\,\mathrm{sgn}\,\beta_m\bigr)\\
&=\tfrac{1}{2}\,\beta_m^2+\tfrac{1}{2}\,\|y^\delta\|_{\ell^2}^2-\beta_m^2\\
&=-\tfrac{1}{2}\,\beta_m^2+\tfrac{1}{2}\,\|y^\delta\|_{\ell^2}^2.
\end{align}
The candidate with greatest $\beta_m^2$ is the true minimizer (which exists by assumption). From the definition of $\beta_m$ we see that $\beta_m^2$ is maximal (w.\,r.\,t.\ $m$) if $|\langle y,\zeta^{(m)}\rangle|$ is maximal. In particular, the existence of a minimizer implies that there is some $m$ with $|\langle y^\delta,\zeta^{(m)}\rangle|\geq|\langle y^\delta,\zeta^{(\tilde{m})}\rangle|$ for all $\tilde{m}\in\mathbb{N}$. If $(\zeta^{(k_l)})_{k_l\in\mathbb{N}}$ is a sequence converging to $\frac{y^\delta}{\|y^\delta\|_{\ell^2}}$, then we have
\begin{equation}
|\langle y^\delta,\zeta^{(m)}\rangle|\geq|\langle y^\delta,\zeta^{(k_l)}\rangle|\to\|y^\delta\|_{\ell^2}.
\end{equation}
Thus, $|\langle y^\delta,\zeta^{(m)}\rangle|=\|y^\delta\|_{\ell^2}$, which implies $y^\delta=\lambda\,\zeta^{(m)}$ for some $\lambda\in\mathbb{R}$. Note that $|\beta_m-\langle y^\delta,\zeta^{(m)}\rangle|=\alpha$ and $|\langle\zeta^{(m)},\zeta^{(l)}\rangle|\leq 1$ for all $l$, which implies the second condition in \eqref{eq:opt_beta}.
\par
It remains to discuss the case of two non-zero compentents in $x_\alpha^\delta$. Let $x_\alpha^\delta=\beta_m\,e^{(m)}+\gamma_m\,e^{(n)}$ be such a minimizer. Remember $\zeta^{(n)}=-\zeta^{(m)}$ as well as $\mathrm{sgn}\,\beta_m\neq\mathrm{sgn}\,\gamma_m$. The former results in
\begin{equation}
A\,x_\alpha^\delta=(\beta_m-\gamma_m)\,\zeta^{(m)},
\end{equation}
the latter in
\begin{equation}
\|x_\alpha^\delta\|_{\ell^1}=(\mathrm{sgn}\,\beta_m)\,(\beta_m-\gamma_m).
\end{equation}
In full analogy to the one-component case one obtains a family of candidates parametrized by $m$ and $\gamma$. Looking at the Tikhonov functional for those candidates one sees that the functional's value does not depend on $\gamma$ but only on $m$ in the same way as for the one-component case. Analogously to the one-component case one obtains the structure of $y^\delta$ allowing for minimizers with two non-zero components.
\end{proof}

Theorem~\ref{thm:l1failure}  shows that for Mazur-type operators Tikhonov regularized solutions only exist for $y^\delta$ from a dense subset of the data space. In principle, one could approximate each $y^\delta\in\ell^2$ by a sequence $y^{(k)}$ with corresponding Tikhonov minimizers $x^{(k)}$. If this sequence converges (at least weakly*), one could consider its limit as a regularized solution for $y^\delta$. But from the structure of the minimizers provided by the proposition above one easily sees that the sequence of minimizers always converges weakly* to $0$ (and it is not norm-Cauchy).

The authors conjecture that in $\ell^1$ there is an intimate connection between both lacking weak*-to-weak continuity and uncomplemented null-spaces, driven by the fact that they were unable to find weak*-to-weak continuous operators with uncomplemented null-space. Both classes of operators with uncomplemented null-space, Mazur-type operators as well as quotient maps, are not weak*-to-weak continuous by Theorem~\ref{thm:newl1}. The following theorem shows that even compositions with Mazur-type operators (cf.\ Examples~\ref{ex:2} and~\ref{ex:3}) always lack weak*-to-weak continuity.

\begin{theorem} \label{thm:Bcomp}
Let $A_{\mathrm{Maz}}:\ell^1\rightarrow Y$ be of Mazur-type with $Y$ some separable Banach space (cf.\ Example~\ref{ex:1}) and let $C:Y\rightarrow Z$ with $C\neq 0$ be some bounded linear operator mapping into a Banach space $Z$. Then $A:=C\circ A_{\mathrm{Maz}}$ is not weak*-to-weak continuous.
\end{theorem}

\begin{proof}
For $n\in\mathbb{N}$ consider $e^{(n)}$ with $1$ at position $n$ and zeros else. The sequence $(e^{(n)})_{n\in\mathbb{N}}$ converges weakly* to zero in $\ell^1$. We show that there is some $\eta\in Z^\ast$ such that the dual product $\langle\eta,A\,e^{(n)}\rangle_{Z^*,Z}$ does not tend to zero as $n \to \infty$, which would mean that $(A\,e^{(n)})_{n\in\mathbb{N}}$ does not converge weakly to zero in $Z$ and would prove the theorem. Now, for arbitrary $\eta\in Z^\ast$, we have
\begin{equation}
\langle\eta,A\,e^{(n)}\rangle_{Z^*,Z}
=\sum_{k=1}^\infty e^{(n)}_k\,\langle\eta,C\,\zeta^{(k)}\rangle_{Z^*,Z}
=\langle\eta,C\,\zeta^{(n)}\rangle_{Z^*,Z}
=\langle C^\ast\,\eta,\zeta^{(n)}\rangle_{Y^*,Y}.
\end{equation}
Take some $\eta$ with $C^\ast\,\eta\neq 0$ and some $y\in Y$ with $\langle C^\ast\,\eta,y\rangle_{Y^*,Y}\neq 0$ and $\|y\|_Y=1$. Then we find a sequence $(k_n)_{n\in\mathbb{N}}$ such that there is norm convergence $\zeta^{(k_n)}\to y$ in $Y$ as $n \to \infty$. Thus
\begin{equation}
\langle C^\ast\,\eta,\zeta^{(k_n)}\rangle_{Y^*,Y}\to\langle C^\ast\,\eta,y\rangle_{Y^*,Y}\neq 0,
\end{equation}
and the proof is complete.
\end{proof}

Weak*-to-weak continuity would imply weak* closedness of the null-space, which in non-reflexive Banach spaces is a stronger property than norm or weak closedness (note that a subspace is norm closed if and only if it is weakly closed). Does weak* closedness of the null-space imply its complementedness in $\ell^1$?
In reflexive Banach spaces weak* and weak convergence coincide and all bounded linear operators are weak*-to-weak continuous. Nethertheless, complemented as well as uncomplemented null-spaces may occur and there is no relationship between complementedness of null-spaces and weak*-to-weak continuity.
In $\ell^\infty$ it is known that uncomplemented null-spaces may occur even for weak*-to-weak continuous operators, see \cite[Theorem~3.1]{Tan18}.
In $\ell^1$ the question whether weak*-to-weak continuity implies complementedness of the null-space remains open. At least the converse is not true.
The following example shows that there are ill-posed operators with complemented null-space which are not weak*-to-weak continuous.

\begin{example}[Not weak*-to-weak continuous injective operator]
Let $A:\ell^1\rightarrow\ell^2$ be defined by
\begin{equation}
A\,x:=\left(\sum_{l=1}^\infty x_l,\,\tfrac{1}{2}\,x_2,\,\ldots,\,\tfrac{1}{k}\,x_k,\,\ldots\right).
\end{equation}
Then $A$ is injective with inverse $A^{-1}:\mathcal{R}(A)\rightarrow\ell^1$ given by
\begin{equation}
A^{-1}\,y:=\left(y_1-\sum_{l=2}^\infty l\,y_l,\,2\,y_2,\,\ldots,\,k\,y_k,\,\ldots\right).
\end{equation}
The inverse is unbounded, because $\|A^{-1}\,e^{(k)}\|_{\ell^1}=2\,k\to\infty$.
The mapping $A$ is not weak*-to-weak continuous, because $A\,e^{(k)}=(1,0,\ldots,0,\frac{1}{k},0,\ldots)$ converges weakly to $e^{(1)}$ and not to $0$.
\end{example}

\end{document}